\documentclass[11pt]{amsart}

\usepackage{amsfonts, amstext, amsmath, amsthm, amscd, amssymb}
\usepackage{epsfig, graphics, color}

\let\oldmarginpar\marginpar
\renewcommand\marginpar[1]{\oldmarginpar[\raggedleft\footnotesize #1]%
{\raggedright\footnotesize #1}}

\renewcommand{\setminus}{{\smallsetminus}}

\newcommand{\HH}{{\mathbb{H}}}

\newcommand{\CC}{{\mathbb{C}}}

\newcommand{\bdy}{{\partial}}
\newcommand{\vol}{{\rm Volume}}
\newcommand{\area}{{\rm Area}}

\newcommand{\length}{{\rm Length}}

\newcommand{\Cr}{{\rm cr}}
\newcommand{\tw}{{\rm tw}}

\def\co{\colon\thinspace}

\theoremstyle{plain}
\newtheorem{theorem}{Theorem}[section]
\newtheorem{corollary}[theorem]{Corollary}
\newtheorem{lemma}[theorem]{Lemma}

\newtheorem{prop}[theorem]{Proposition}

\newtheorem*{namedtheorem}{\theoremname}
\newcommand{\theoremname}{testing}
\newenvironment{named}[1]{\renewcommand{\theoremname}{#1}\begin{namedtheorem}}{\end{namedtheorem}}

\theoremstyle{definition}
\newtheorem{define}[theorem]{Definition}

\begin{document}
\title{Cusp volumes of alternating knots}
\author{Marc Lackenby}
\address{Mathematical Institute, University of Oxford, Woodstock Road, Oxford, OX2 6GG, UK}

\author{Jessica S. Purcell}

\address{School of Mathematical Sciences, 9 Rainforest Walk, Monash University, Victoria 3800, Australia}



\begin{abstract}
We show that the cusp volume of a hyperbolic alternating knot can be bounded above and below in terms of the twist number of an alternating diagram of the knot. This leads to diagrammatic estimates on lengths of slopes,
and has some applications to Dehn surgery. Another consequence is that there is a
universal lower bound on the cusp density of hyperbolic alternating knots.
\end{abstract}

\maketitle

\section{Introduction}\label{sec:intro}
A major goal in knot theory is to relate geometric properties of the knot complement to combinatorial properties of the diagram.  In this paper, we relate the diagram of a hyperbolic alternating knot to its
cusp area, answering a question asked by Thistlethwaite~\cite{thistle:length}. He observed using computer software such as SnapPea~\cite{snappea, snappy} and other programs, including one developed by Thistlethwaite and Tsvietkova~\cite{thistle-tsviet}, that all alternating knots with small numbers of crossings have cusp area related to their twist number.  This led him to ask if all alternating knots have cusp area bounded below by a linear function of their twist number.  This question is finally settled in the affirmative in this paper.  Specifically, we show the following.

\begin{theorem}\label{thm:main}
Let $D$ be a prime, twist reduced alternating diagram for some hyperbolic knot $K$, and let $\tw(D)$ be its twist number. 
Let $C$ denote the maximal cusp of the hyperbolic 3-manifold $S^3\setminus K$. Then 
\[ A\,(\tw(D)-2) \leq \area(\bdy C) < 10\sqrt{3}\,(\tw(D)-1), \]
where $A$ is at least $2.278 \times 10^{-19}$.
\end{theorem}

The terms twist reduced, twist number, and maximal cusp are defined below.

Theorem \ref{thm:main} has the following almost immediate corollary concerning the length of slopes on a cusp of an alternating knot.

\begin{corollary}\label{cor:longitude}
Every slope, except possibly the meridian, on the maximal cusp of a hyperbolic alternating knot has length at least $B\,(\tw(D)-2)$,
where $B$ is at least $7.593 \times 10^{-20}$ and $D$ is a prime, twist reduced alternating diagram of the knot.
\end{corollary}

This result has several applications. We mention the following, purely topological consequence.

\begin{theorem}\label{thm:surgeryfiniteness}
For any closed 3-manifold $M$ with sufficiently large Gromov norm, there are at most finitely many prime alternating knots $K$
and fractions $p/q$ such that $M$ is obtained by $p/q$ surgery along $K$.
\end{theorem}

This is related to Problem~3.6 in Kirby's list \cite{kirby:problems}, which is attributed to Clark: is there a closed 3--manifold $M$ that can be obtained
by Dehn surgery along infinitely many distinct knots in the 3--sphere? This was answered affirmatively by Osoinach~\cite{osoinach:surgery}, who provided examples of such 3--manifolds $M$ (see also \cite{abe-jong-luecke-osoinach}). However, Theorem~\ref{thm:surgeryfiniteness} gives a negative
answer to the version of Kirby's problem that restricts to prime alternating knots and 3--manifolds with large Gromov norm.

\subsection{Alternating diagrams}
We now recall various definitions concerning alternating diagrams that we will use throughout the paper.

In a diagram of a knot or link, a \emph{twist   region} is a (possibly empty) string of bigon regions in the 4--valent diagram graph, such that the knot is alternating within this region, and such that the string is maximal, in the sense that no other bigons of the diagram lie on either side.  A single crossing adjacent to no bigons is also a twist region.

The \emph{twist number} $\tw(D)$ of a diagram $D$ is its number of twist regions. See Figure \ref{fig:twistdefn}.

\begin{figure}
  \includegraphics[width=2in]{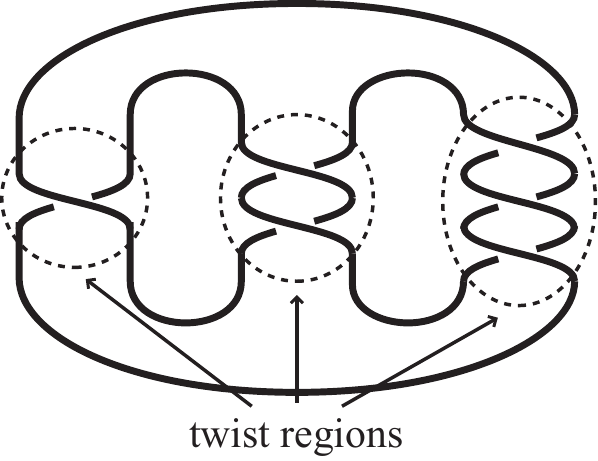}
  \caption{A diagram with twist number $3$. Figure adapted from \cite{lackenby:volume-alt}}
  \label{fig:twistdefn}
\end{figure}

Two crossings are \emph{twist equivalent} if there is a simple closed curve in the projection plane of the diagram
that meets the diagram exactly at the two crossings, and at each of these crossings, runs between opposite regions.

Note that we can apply a sequence of flypes to an alternating diagram to put all twist equivalent crossings together into a single twist region.  When this is done, we say the diagram is \emph{twist   reduced}.  More precisely, we say a diagram is \emph{twist reduced} if for any simple closed curve meeting the diagram at precisely two crossings, and such that, at each of these crossings, it runs between opposite regions, the curve encloses a string of bigons on one side.  By the above observation, every alternating diagram has a twist reduced alternating diagram. See Figure \ref{fig:twistred}.

\begin{figure}
  \includegraphics{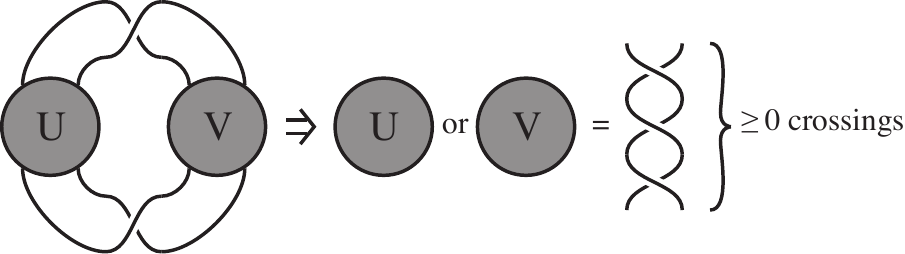}
  \caption{A twist reduced diagram. Figure from \cite{lackenby:volume-alt}}
  \label{fig:twistred}
\end{figure}

The \emph{twist number} $\tw(K)$ of a prime alternating knot $K$ is the number of twist regions in a prime, twist reduced diagram.  
The twist number can be seen to be a well-defined invariant of an alternating knot in several ways, for example by using the invariance of `characteristic squares' under flyping as in Section 4 of \cite{lackenby:volume-alt} and then applying the solution to the Tait Flying Conjecture by Menasco and Thistlethwaite \cite{menasco-thist:classif-alt}, or by relating twist number to the Jones polynomial as in \cite{dasbach-lin:volumish}.

\subsection{Cusps}

Menasco showed that any nonsplit prime alternating link that is not a $(2,q)$--torus link is hyperbolic \cite{menasco}, using results of Thurston on the geometrization of Haken manifolds \cite{thurston:bulletin}.  One important feature of a hyperbolic knot or link complement is its cusp(s), and several key geometric invariants arise from the study of cusps of hyperbolic 3--manifolds.

\begin{define}\label{def:cusps}
If $M$ is a finite-volume hyperbolic 3-manifold, then it has ends of the form $T^2 \times [1, \infty)$.
Each end can be realised geometrically as the image $p(H)$ of some horoball $H$ in $\mathbb{H}^3$,
where $p \colon \mathbb{H}^3 \rightarrow M$ is the covering map. This is known as a \emph{cusp}.
For each end, there is a 1-parameter family of cusps that are obtained from each other by 
changing the choice of horoball $H$, but keeping the same limiting point on the sphere at infinity.
If the cusps are expanded in the hyperbolic manifold until each is tangent to itself or another, 
the result is called a choice of \emph{maximal cusps}.  When the manifold has just one end, 
for example in the case of a hyperbolic knot complement, then there is a unique maximal cusp.  
When we speak of \emph{the} cusp volume of a hyperbolic alternating knot, we mean the volume of this maximal cusp. 
Similarly, the \emph{cusp area} is the Euclidean area of the torus that is the boundary of the maximal cusp.
By an exercise in hyperbolic geometry, the cusp area is twice the cusp volume.

Finally, the \emph{cusp density} of $M$ is defined to be the largest possible volume in a choice of
maximal cusps divided by the total volume of the manifold.
\end{define}

\subsection{Cusp density}

Our results have interesting consequences for cusp density.  By work of B\"or\"oczky \cite{boroczky}, the cusp density of a hyperbolic knot or link complement is at most $\sqrt{3}/(2v_3)$, where $v_3 \approx 1.01494$ is the volume of an ideal regular tetrahedron. The cusp density of the figure--8 knot realizes this bound \cite{thurston}.  It follows from work of Eudave-Mu\~noz and Luecke that general knots in $S^3$ may have arbitrarily small cusp density \cite{eudave-munoz-luecke}.  However, Theorem \ref{thm:main} implies that alternating knots have universally bounded cusp density.

\begin{corollary}\label{cor:cusp-density}
The cusp density of an alternating knot has a universal lower bound.  
\end{corollary}

This is a consequence of Theorem \ref{thm:main}, because the cusp volume of $S^3 \setminus K$ is at least $(A/2)(\tw(K)-2)$,
whereas it is a theorem of Lackenby, with improvements due to Agol and D. Thurston \cite{lackenby:volume-alt},
that the hyperbolic volume of $S^3 \setminus K$ is at most $10 v_3 (\tw(K) - 1)$.

\subsection{Slope length}

\begin{define}
A \emph{slope} on a torus is an isotopy class of essential simple closed curves. If $M$ is a compact 3-manifold
with interior admitting a complete, finite volume hyperbolic structure, and $C$ is a choice of maximal cusps,
then a slope $\sigma$ on $\partial M$ has an associated \emph{length}. This is defined to be the length of any
Euclidean geodesic representative for $\sigma$ on $\partial C$.
\end{define}

Cusp area and slope length are closely related, since high cusp area forces every slope, except possibly one,
to be long.

In the case of the meridian slope for a knot in the 3-sphere, its length is well known to be bounded, but nevertheless,
several interesting questions about meridian length remain open. It was shown by Thurston \cite{thurston} that
with respect to any maximal cusp, the length of any slope is at least one.
 By work of Agol and Lackenby \cite{agol:bounds, lackenby:word}, the length of the meridian of a knot in $S^3$ can be no more than six.  For an alternating knot, the meridian has length less than three \cite{adams-students:cusp-area}.  In fact, no alternating knots have been found with meridian length more than two \cite{adams:hyperbolic-knots}.

While the meridian of a knot in $S^3$ has universally bounded length, the next shortest curve on a cusp does not.  Several different results in the past gave evidence that for alternating knots, its length is bounded below by the twist number.  For example, Lackenby showed that, in a combinatorial metric on the cusp torus, the second shortest curve had length bounded below by a linear function of the twist number \cite{lackenby:word}.  Unfortunately, the combinatorial metric that Lackenby used cannot be applied to the Euclidean metric on the torus in any known way, and so Lackenby's result does not give geometric information on the cusp.  Purcell showed that for knots with a high number of crossings in each twist region, including alternating knots, the second shortest curve had length bounded below by a constant times the twist number \cite{purcell:cusps}.  However, this result gave no information about knots with a small number of crossings in any twist region.  Futer, Kalfagianni, and Purcell showed the cusp area is bounded below by twist number for 2--bridge knots \cite{fkp:farey}, which are a subset of alternating knots.

Corollary \ref{cor:longitude} gives that the length of every non-meridional slope for a hyperbolic alternating knot must be long, unless the knot lies in one of the limited families with bounded twist number. Corollary~\ref{cor:longitude} is an immediate consequence of Theorem~\ref{thm:main}, as follows. For suppose that $\sigma$ is some non-meridional slope on the boundary of a maximal cusp $C$, and let $\mu$ be the meridian. Let $L(\sigma)$ and $L(\mu)$ be their lengths, and denote the modulus of their intersection number by $\Delta(\sigma, \mu)$. Then, an elementary calculation in Euclidean geometry (for example as in the proof of Lemma~2.1 in \cite{cooper-lackenby}) gives that
\[ L(\sigma) \ L(\mu) \geq \area(\bdy C) \ \Delta(\sigma, \mu). \]
The length of $\mu$ is at most $3$, by \cite{adams-students:cusp-area}, and so we deduce that
\[ L(\sigma) \geq \frac{ \area(\bdy C) \ \Delta(\sigma, \mu) }{ L(\mu) } \geq B \, (\tw(K)-2), \]
where $B$ is greater than $7.593 \times 10^{-20}$.

\subsection{Applications to Dehn surgery}

The above results on slope length are significant because slope length has important consequences for Dehn surgery. It is a well--known consequence of the 6--theorem of Agol \cite{agol:bounds} and Lackenby \cite{lackenby:word}, together with Perelman's proof of the geometrisation conjecture, that if $M$ is a finite--volume hyperbolic 3--manifold and $s_1, \dots, s_k$ are slopes on distinct components of $\partial M$, such that, with respect to some horoball neighbourhood $C$ of the cusps, each $s_i$ has length more than 6, then the manifold $M(s_1, \dots, s_k)$ obtained by Dehn filling along these slopes is hyperbolic. It was shown by Futer, Kalfagianni and Purcell \cite{fkp:fillingvolume} that if these slopes all have length at least $\ell_{\rm min} > 2 \pi$ on $\partial C$, then the hyperbolic volume of
$M(s_1, \dots, s_k)$  satisfies
\[ \vol(M(s_1, \dots, s_k)) \geq \left (1 - \left ( \frac{2 \pi}{ \ell_{\rm min}} \right )^2 \right ) ^{3/2}
\vol(M). \]
Therefore, applying Corollary \ref{cor:longitude}, we deduce the  second inequality in the following result.

\begin{theorem}\label{thm:volumesurgery}
Let $K$ be a hyperbolic alternating knot. If $\tw(K) \geq 1.361 \times 10^{20}$, then any manifold $M$ obtained by a non-trivial surgery along $K$ is hyperbolic and satisfies
\[
\frac{v_8}{2}(\tw(K)/2 -1) \leq \frac{1}{2} \vol(S^3 \setminus K)
\leq \vol(M) < \vol(S^3 \setminus K) \leq 10 v_3 (\tw(K) - 1),
\]
where $v_8 \approx 3.66386$ is the volume of regular hyperbolic ideal octahedron, and $v_3 \approx 1.01494$ is the volume of a regular hyperbolic ideal tetrahedron.
\end{theorem}

Two of the remaining inequalities are upper and lower bounds on the volume of hyperbolic alternating link complements, due to Lackenby \cite{lackenby:volume-alt}, with improvements by Agol and D. Thurston \cite{lackenby:volume-alt}, and Agol, Storm and W. Thurston \cite{agol-storm-thurston}. The third inequality is the fact that hyperbolic volume decreases when Dehn filling is performed \cite{thurston}.

We can also deduce Theorem \ref{thm:surgeryfiniteness} from Theorem \ref{thm:main}, together with a theorem of Cooper and Lackenby \cite{cooper-lackenby}, which asserts that, for any $\epsilon > 0$ and any closed orientable 3-manifold $M$, there are at most finitely many cusped hyperbolic 3-manifolds $X$ and slopes $s$ on $\partial X$ with length at least $2 \pi + \epsilon$, such that $M$ is obtained by Dehn filling $X$ along $s$. 

Let $M$ be a closed 3--manifold with Gromov norm more than $8.561 \times 10^{20}$. Suppose that $M$ is obtained by $p/q$ surgery along a prime alternating knot $K$. Let $X$ be the exterior of $K$. Then $X$ is hyperbolic, and since Gromov norm does not increase when Dehn filling is performed, the Gromov norm of $X$ is also at least $8.561 \times 10^{20}$. Since Gromov norm and hyperbolic volume are proportional, the volume of $X$ is at least $8.561 \times 10^{20}  v_3$. Hence, by Lackenby's theorem, with improvements due to Agol and D.~Thurston~\cite{lackenby:volume-alt}, the twist number of $K$ is at least $8.561 \times 10^{19}$. So, by Corollary~\ref{cor:longitude}, the length of the filling slope $p/q$ is at least $(7.593 \times 10^{-20}) \times (8.561 \times 10^{19}-2) > 6.5$. Setting $\epsilon = 6.5 - 2\pi$, and applying the theorem of Cooper and Lackenby, we obtain Theorem~\ref{thm:surgeryfiniteness}.

\subsection{Crossing arcs}

In his experimental analysis of the hyperbolic structures of alternating knot complements, 
Thistlethwaite observed several geometric features
that are related to cusp area, and formulated various conjectures. To explain one of these conjectures, we need the following definitions.

\begin{define}\label{def:length-arc}
Let $\alpha$ be an arc in $S^3$, such that $\alpha \cap K = \partial \alpha$. Suppose that $\alpha$ is homotopic in $S^3\setminus K$ to a geodesic $\gamma$. A maximal cusp for $S^3 \setminus K$ intersects $\gamma$ in two half-open intervals which contain the ends of $\gamma$, together possibly with some closed intervals. Define the \emph{length} of $\alpha$ to be the length of the arc that results from removing the two half-open intervals from $\gamma$.
\end{define}

\begin{define}\label{def:crossing-arc}
A \emph{crossing arc} is defined to be an embedded arc in $S^3$ with endpoints on $K$, which projects to a single point lying at a crossing in the diagram of $K$.
\end{define}

Thistlethwaite conjectured that crossing arcs in alternating knots have universally bounded length, with length bounded above by $\log 8$ \cite{thistle:length}.  Since two crossing arcs are homotopic in $S^3\setminus K$ only if they belong to the same twist region (as shown in Proposition 7.12 of \cite{lp:twcheck}), and since short arcs contribute to cusp area, Thistlethwaite's conjecture would imply Theorem \ref{thm:main}.  In fact, our proof of Theorem \ref{thm:main} hinges on the fact that many distinct geodesic arcs in $S^3 \setminus K$ have bounded length. However, we cannot show that crossing arcs are short, but expect this to be the case. Therefore, we offer Theorem \ref{thm:main} as additional evidence for Thistlethwaite's conjecture.

\subsection{Organization}
The results in this paper are organized as follows.  In Section~\ref{sec:bounded}, we prove the main result for alternating knots with a bounded number of crossings per twist region.  This case is simpler to explain than the general case, and yet introduces the main tools and techniques we will use in general, particularly Theorem~\ref{thm:arcs}.  We apply this theorem to the checkerboard surfaces of the diagram.  In Section~\ref{sec:twisted}, we remind the reader of an essential immersed surface, originally appearing in~\cite{lp:twcheck}, which will play the role of the checkerboard surfaces in the general case.  We use this to give a proof of Theorem~\ref{thm:main}. 

\subsection{Acknowledgements}
Purcell is supported in part by NSF grant DMS--1252687, and by a Sloan Research Fellowship.

\section{Bounded crossings in each twist region}\label{sec:bounded}
In this section, we prove a version of the main theorem in the case that our alternating knot has a bounded number of crossings per twist region.  That is, we show that such knots have cusp volume bounded above and below by linear functions of the twist number.

We work through this case separately because it is somewhat simpler than the most general case, and also because it allows us to develop and apply many of the tools we will use in the general setting.

\subsection{Checkerboard surfaces}

Recall the checkerboard surfaces of an alternating link complement are obtained by coloring the regions of the diagram in a checkerboard manner, in blue and red, say, and then attaching twisted bands at crossings to connect regions of the same color.  Note that the checkerboard surfaces intersect only at crossings, and then in a single arc running from the strand of the link at the top of the crossing to the strand of the link at the bottom. This is a crossing arc, as in Definition \ref{def:crossing-arc}.

The following result is due to Menasco and Thistlethwaite \cite{menasco-thist:classif-alt}.  

\begin{lemma}\label{lemma:checkerboard-essential}
The checkerboard surfaces arising from a prime alternating diagram are essential.  
\end{lemma}

Any essential surface in a hyperbolic 3--manifold can be given a pleating (see Thurston \cite{thurston}, or for a more detailed proof, see Canary, Epstein, and Green \cite{canary-epstein-green}, or Lackenby \cite{lackenby:word}).  We pleat both checkerboard surfaces in $S^3\setminus K$.  This gives an induced hyperbolic metric on the disjoint union of the two checkerboard surfaces.

\begin{lemma}\label{lemma:checkerboard-area}
Let $S$ denote the disjoint union of the two checkerboard surfaces for the alternating knot $K$ with a prime, twist reduced diagram, which we also denote by $K$.  Then the Euler characteristic of $S$ is $\chi(S) = 2-\Cr(K)$,
where $\Cr(K)$ is the number of crossings of the diagram. 
Thus under the hyperbolic metric obtained from pleating $S$ in $S^3 \setminus K$, the area of $S$ is $\area(S) = 2\pi(\Cr(K) - 2)$.
\end{lemma}

This result is not hard, and appears elsewhere (see, for example \cite{adams-students:cusp-area}).  We offer a particularly simple proof.  

\begin{proof}
In the disjoint union of red and blue surfaces, each crossing arc appears twice: once in the red, and once in the blue.  In $S^3 \setminus K$, replace each of the crossing arcs with two homotopic arcs, each with an endpoint on the overstrand of the crossing, but with separated endpoints on the understrand of the crossing.  Attach red and blue regions of the diagram to these arcs.  The result is just the plane of projection of the diagram, punctured once by each overcrossing, hence has Euler characteristic $2-\Cr(K)$.  But by construction, this Euler characteristic is exactly the Euler characteristic of the disjoint union of the two checkerboard surfaces.  The area formula follows immediately from the Gauss--Bonnet theorem.
\end{proof}

\begin{lemma}\label{lemma:union-cusparea}
Let $S$ denote the disjoint union of two essential properly immersed surfaces $S_1$ and $S_2$ in $S^3\setminus K$. Give $S$ a hyperbolic metric, by pleating $S$ and pulling back the path metric on $S^3\setminus K$. Let $H$ denote the maximal cusp of $S^3\setminus K$.  Suppose that, for $i =1$ and $2$, the boundary of $S_i$ is a single curve on $\partial N(K)$ with integral slope $\sigma_i$.  
Then $S$ has an embedded cusp $H_S$ with cusp area at least $|\sigma_1-\sigma_2|$, such that $H_S$ maps into $H$ under the map from $S$ to $S^3\setminus K$.  Moreover, if $K$ is neither the figure-eight knot nor the knot $5_2$, the area of $H_S$ is at least $\root 4 \of 2 \, |\sigma_1-\sigma_2|$.
\end{lemma}

\begin{proof}
For $i=1, 2$, there exists an embedded cusp $H_i$ for $S_i$ with area at least the length of a geodesic representative $\gamma_i$ for $\partial S_i$ on $\partial H$, and which has image in $S^3\setminus K$ contained in $H$, for example by a lemma of Futer and Schleimer \cite[Lemma~2.5]{futer-schleimer}.

Orient $\gamma_1$ and $\gamma_2$ so that their intersection numbers with the meridian of $K$ are opposite.  If we resolve the intersections between $\gamma_1$ and $\gamma_2$ so that they respect these orientations, the result is a collection of meridians with the same length as $\gamma_1 \cup \gamma_2$.  The number of these meridians is $|\sigma_1-\sigma_2|$.  Each has length at least $1$, by standard arguments in hyperbolic geometry.  So, the length of $\gamma_1 \cup \gamma_2$ is at least $|\sigma_1-\sigma_2|$, and this forms a lower bound for the area of $H_S = H_1 \cup H_2$.

In fact, Adams \cite{adams:waist2} proved that the meridian of any hyperbolic knot in the 3-sphere has length at least $\root 4 \of 2$, with the exception of the figure-eight and $5_2$ knots. Hence, using this bound, the length of $\gamma_1 \cup \gamma_2$ is at least $\root 4 \of 2 \, |\sigma_1-\sigma_2|$, and again this forms a lower bound for the area of $H_S = H_1 \cup H_2$.
\end{proof}

\begin{lemma}\label{lemma:checkerboard-cusparea}
Let $S$ denote the disjoint union of the two checkerboard surfaces for the hyperbolic alternating knot $K$.  Then $S$ has an embedded cusp $H_S$ such that the area of $H_S$ is at least twice the number of crossings in the prime, twist reduced diagram of $K$, and such that $H_S$ maps into the maximal cusp for $S^3 \setminus K$.  That is, 
$$\area(H_S) \geq 2 \, \Cr(K).$$
Moreover, if $K$ is neither the figure-eight knot nor the knot $5_2$, the area of $H_S$ is at least $2^{5/4} \, \Cr(K)$.
\end{lemma}

\begin{proof}
Color the two checkerboard surfaces red and blue, and denote them $R$ and $B$.  Denote their disjoint union by $S$.  Note that $R$ and $B$ are properly embedded in the exterior of $K$, each with boundary a single curve with integral slope. 

We claim that the difference in slopes  of $\partial R$ and $\partial B$ is twice the crossing number.  This is because $R$ and $B$ intersect exactly in crossing arcs, with boundary slopes intersecting once at an overcrossing, once at an undercrossing.  See Figure \ref{fig:crossingandsurface}.

\begin{figure}
  \includegraphics{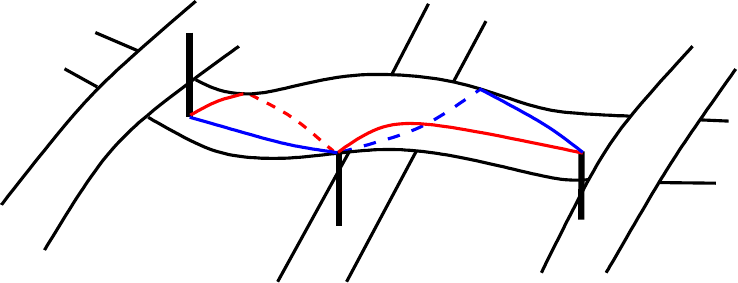}
  \caption{The red and blue surfaces run over two meridians per
    over--crossing.}
  \label{fig:crossingandsurface}
\end{figure}

Thus Lemma~\ref{lemma:union-cusparea} gives the required cusp for $S$, with area at
least $2\,\Cr(K)$. Moreover, when $K$ is neither the figure-eight knot nor the knot $5_2$, the area of the cusp is at least $2^{5/4} \,\Cr(K)$.
\end{proof}

The above lemma is simple, but quite striking. It is easy to construct hyperbolic surfaces with large area but with small cusp area, but Lemma \ref{lemma:checkerboard-cusparea} asserts that this type of geometry never occurs among checkerboard surfaces of alternating knots, when in pleated form.

This lower bound on the cusp area of $S$ is central to our argument. We will use it in Subsection~\ref{subsec:shortarcs} to establish the existence a large collection of disjoint properly embedded essential arcs in $S$, each of which has bounded length. By Lemma~\ref{lemma:checkerboard-essential}, these can be homotoped to geodesics in $S^3 \setminus K$ without increasing their length. We will see in Subsection~\ref{subsec:shortarcs} that each such arc gives a definite contribution to the cusp volume of $S^3 \setminus K$. However, it may be the case that distinct geodesics in $S$ are homotoped to the same geodesic in $S^3 \setminus K$.  In the remainder of this subsection, we investigate this phenomenon more thoroughly.

Consider a twist region of the diagram. A small regular neighbourhood of this twist region in the diagram is a disc, and a regular neighbourhood of this disc is a 3--ball in $S^3$.  The intersection between this ball and a checkerboard surface is the subsurface that is \emph{associated} with this twist region. Note that if the twist region has $c$ crossings, then one of these subsurfaces has Euler characteristic $2-c$ and the other is a disc.  See Figure \ref{fig:twistsubsurface}.

\begin{figure}
  \includegraphics[width=2in]{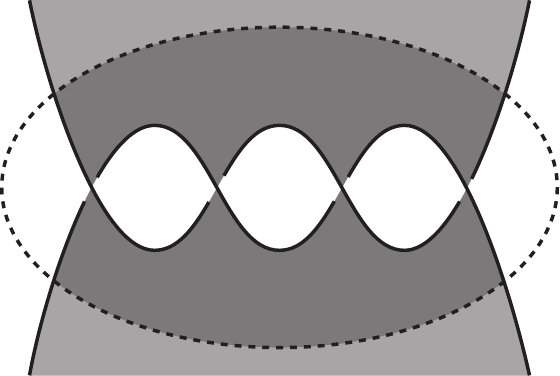}
  \caption{The subsurface of a checkerboard surface associated with a twist region.}
  \label{fig:twistsubsurface}
\end{figure}

The following is Proposition~7.12 of~\cite{lp:twcheck}.

\begin{lemma}\label{lemma:checkerboard-htpcarcs}
Let $S$ denote the disjoint union of the two checkerboard surfaces of a prime, twist reduced, alternating diagram of a hyperbolic knot $K$.  Suppose $a_1$ and $a_2$ are disjoint essential embedded arcs in $S$ that are not homotopic in $S$, but are homotopic in $S^3\setminus K$ after including $S$ into $S^3\setminus K$.  Then either $a_1$ and $a_2$ are isotopic in $S$ to crossing arcs in the same twist region of the diagram, or they both lie on the same checkerboard surface and both are isotopic in that checkerboard surface to arcs in the same subsurface associated with some twist region.
\end{lemma}

The proof of this result does not require all of the machinery of \cite{lp:twcheck}, but mostly just the techniques developed in Section 7 of that paper. The proof is rather reminiscent of the arguments in \cite{lackenby:word} and \cite{lackenby:volume-alt}. Since the arcs $a_1$ and $a_2$ are homotopic in $S^3 \setminus K$, there is a map of rectangle into $S^3 \setminus {\rm int}(N(K))$, with the left and right edges mapped to $a_1$ and $a_2$, and the top and bottom mapped to $\partial N(K)$. One then considers the inverse image of the checkerboard surfaces in this rectangle, which divides the rectangle into regions. Each region is mapped into the complement of the checkerboard surfaces, and so one can view its boundary as specifying a curve in the alternating diagram. Lemma \ref{lemma:checkerboard-htpcarcs} is then proved by analysing these curves and by ruling out various trivial configurations, until the map of the rectangle into $S^3 \setminus {\rm int}(N(K))$ is of a very restricted form. The conclusion of the lemma then follows quickly. For more details, see Section 7 of \cite{lp:twcheck}.

\begin{corollary}\label{corollary:numberhomotopicarcs}
Let $S$ denote the disjoint union of the two checkerboard surfaces for a prime, twist reduced, alternating diagram of the hyperbolic knot $K$. Suppose that each twist region contains at most $N$ crossings. Then, for any collection of disjoint embedded essential non-parallel arcs in $S$, all of which are homotopic in $S^3 \setminus K$, the number of such arcs is most $3N-1$.
\end{corollary}

\begin{proof}
We may assume that, in this collection, there is more than one arc lying in the same checkerboard surface, $B$, say. By Lemma~\ref{lemma:checkerboard-htpcarcs}, all the arcs lying in $B$ can be isotoped to lie in the subsurface associated with some twist region. We may arrange that, after this isotopy, the arcs remain disjoint, embedded and non-parallel. The number of disjoint essential arcs, no two of which are parallel, that can lie in such a subsurface is at most $3N-2$. Also, by Lemma~\ref{lemma:checkerboard-htpcarcs}, any arc in the collection in $R$ can isotoped to lie in the subsurface associated to this twist region. This can support at most one such arc, up to isotopy.
\end{proof}

\subsection{Short arcs and cusp volume}\label{subsec:shortarcs}

The following result will be used to show that there are many geodesic arcs of bounded length in a hyperbolic alternating knot complement.  

Recall that, in Definition \ref{def:length-arc}, we defined the length of an arc in $S^3$, with both endpoints on $K$. We present a variation of this definition now. Let $F$ be a finite--volume hyperbolic surface, and let $H$ be an embedded horoball neighbourhood of its cusps. Let $\alpha$ be a bi-infinite geodesic in $F$, with both its ends in $H$. Then the intersection between $\alpha$ and $H$ consists of two half-open intervals containing these ends, plus perhaps some closed bounded intervals. Define the \emph{length} of $\alpha$ with respect to $H$ to be the length of the arc obtained by removing the two half-open intervals. Thus, if $\alpha$ also intersects $H$ in some closed bounded intervals, then these do contribute to the length of $\alpha$.

\begin{theorem}\label{thm:arcs}
Let $F$ be a (possibly disconnected) finite--area hyperbolic surface.  Let $H$ be an embedded horoball neighborhood of the cusps of $F$.  Let $k = \area(H)/\area(F)$ and let $d>0$.  Then there is a collection of at least 
$$\frac{(k e^d - 1)\pi }{(e^d - 1)(\sinh (d) + 2 \pi)} |\chi(F)|$$
embedded disjoint bi-infinite geodesic arcs, each with both ends in $H$, and each having length at most $2d$ with respect to $H$. \end{theorem}

Before we give the proof of this theorem, we make a couple of observations about it. On the one hand, the lower bound on the number of geodesic arcs that it provides, with length at most $2d$, can be viewed as somewhat inefficient. This is because, as $d$ tends to infinity, the lower bound decreases to zero, whereas clearly the number of geodesics with length at most $2d$ is an increasing function of $d$. However, the real utility of the theorem is its universal nature. The maximal number of disjoint
embedded geodesic arcs in a cusped hyperbolic surface $F$ is $3|\chi(F)|$. Theorem \ref{thm:arcs} asserts that one may find a collection of such arcs, with cardinality which is a definite fraction
of this maximum, and where all the arcs have bounded length. The fraction depends little on the topology and geometry of $F$: just the cusp density of $F$ and the length of the arcs determine this fraction. It is the general nature of this estimate that is crucial for our proof of Theorem \ref{thm:main}.

\begin{proof}
Pick a maximal collection of embedded disjoint bi-infinite geodesic arcs $G$, each with both ends in $H$, and each having length at most $2d$ with respect to $H$. Use the upper half space model of $\HH^2$.  Let $\widetilde{H}$ be the inverse image of $H$ in $\HH^2$.  We may arrange that one component $H_\infty$ of $\widetilde{H}$ is $\{ y \geq 1\}$.  For each component $g$ of $G$, pick lifts $\widetilde g$ and $\widetilde g'$ to $\HH^2$, each of which starts on $H_\infty$, but which do not differ by a covering transformation that preserves $H_\infty$.  These run to other components $H_{\widetilde g}$ and $H_{\widetilde g'}$ of $\widetilde{H}$. Let $I(\widetilde g)$ be an open interval on $\partial H_\infty$ centered at the point where $\widetilde g$ exits, with length $2\pi$.  Define $I(\widetilde g')$ similarly.  Each of these projects to an interval in $\partial H$ with length also $2\pi$ (or possibly to all of $\partial H$). 

For each component $g$ of $G$, its intersection with $H$ is two half-open intervals, plus possibly some closed intervals (which may be points). For each closed interval, pick a lift of this to $\HH^2$ that lies in $H_\infty$. This lift lies in a lift $\widehat{g}$ of $g$ in $\HH^2$. Consider the vertical projection of $\widehat{g} \cap H_\infty$ onto $\partial H_\infty$. Its image in $\partial H$ is termed the \emph{vertical projection of the interval onto $\partial H$}. 

We define $I$ to be the union of the following subsets of $\partial H$:
\begin{enumerate}
\item for each geodesic $g$ in $G$, the images of $I(\widetilde g)$ and $I(\widetilde g')$ in $\partial H$;
\item for each geodesic $g$ in $G$ and for each closed interval of $g \cap H$, the vertical projection of this interval onto $\partial H$.
\end{enumerate}

\medskip
\noindent {\bf Claim 1.} The length of $I$ is at most $|G|(2 \sinh (d) + 4 \pi)$. 
\medskip

Clearly, for each geodesic $g$ in $G$, the intervals $I(\widetilde g)$ and $I(\widetilde g')$ contribute at most $4 \pi$ to the length of $I$. So, we consider the projection of the closed intervals in $g \cap H$ onto $\partial H$. Let $\ell_1, \dots, \ell_n$ be the lengths of these intervals. Then the projections of these intervals onto $\partial H$ have lengths $2 \sinh (\ell_1/2) , \dots, 2 \sinh (\ell_n/2)$.
Since $\sinh$ is increasing and convex on the positive real line, and $\ell_1 + \dots + \ell_n \leq 2d$, we deduce that the total length of these projected intervals is at most $2 \sinh d$. This proves the claim.

\medskip
\noindent {\bf Claim 2.} For each lift $\widehat{g}$ of a component of $G$, if $\widehat{g}$ intersects $\partial H_\infty$, then the vertical projection of $\widehat{g}$ onto $\partial H_\infty$ lies entirely in the inverse image of $I$.
\medskip

We may assume that $\widehat{g}$ has maximal projection onto $\partial H_\infty$, in the sense that this projection is not contained within the projection of some other lift of $G$.  Let $x$ be one of the two points of $\widehat{g} \cap \partial H_\infty$. Let $\overline{\alpha}$ be the geodesic ray starting at $x$ running vertically into $H_\infty$. Then $\overline{\alpha} - x$ is disjoint from the inverse image of $G$, by our assumption about the maximality of the projection of $\widehat{g}$. Let $\overline{\beta}$ denote the subset of $\widehat{g}$ that is a geodesic ray starting at $x$ and exiting $H_\infty$. Then $\overline{\alpha} \cup \overline{\beta}$ is a piecewise geodesic, and after a small isotopy, it can be made disjoint from the inverse image of $G$. Hence, it maps to an arc in $F$ that is disjoint from $G$. This is embedded, because $\overline{\alpha}$ and $\overline{\beta}$ both have images that are embedded in $F$, and these images have disjoint interiors. Straighten $\overline{\alpha} \cup \overline{\beta}$ to a geodesic $\gamma$ that hits $\partial H_\infty$ orthogonally. Then the image of $\gamma$ is also embedded and is isotopic to an arc disjoint from $G$. Its length outside the cusp is less than $2d$. So, the image of $\gamma$ lies in $G$, by the maximality of $G$. Therefore $I(\gamma)$ maps to $I$.

Now the triangle with sides $\overline{\alpha}$, $\overline{\beta}$ and $\gamma$ has area less than $\pi$. On the other hand, the area contained in the portion of $H_\infty$ lying over $I(\gamma)$ on the side of $\gamma$ containing $x$ is $\pi$. Hence $I(\gamma)$ contains the interval in $\partial H_\infty$ between $\gamma\cap \partial H_\infty$ and $x$.
Hence, this entire interval lies in the inverse image of $I$. Since this argument applies to both points of $\widehat{g} \cap \partial H_\infty$, the claim is proved.

For the parts of $\partial H$ not lying in $I$, expand $H$ by vertical distance $d$, giving a (closed) neighbourhood $H'$ of the cusps.  See Figure~\ref{fig:horoballexpand1}.

\begin{figure}
  \includegraphics[width=5in]{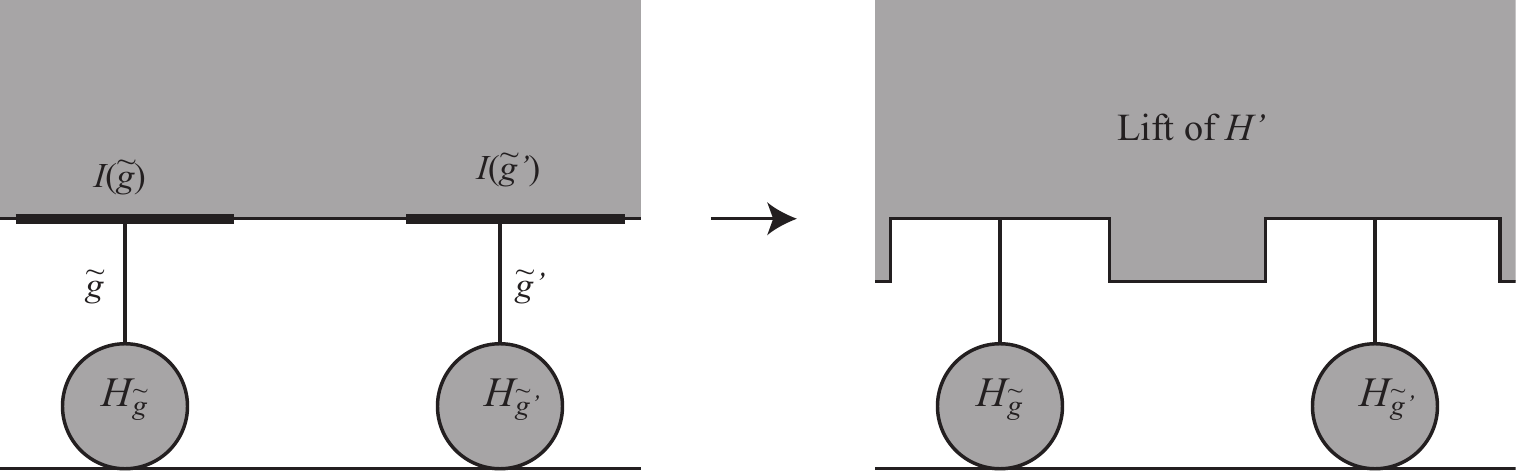}
  \caption{Expanding $H$ to form $H'$}
  \label{fig:horoballexpand1}
\end{figure}

\medskip
\noindent {\bf Claim 3.} $H'$ is disjoint from $G-H$ and is embedded in $F$.
\medskip

Suppose that $H'$ intersects $G-H$ or is not embedded in $F$. We view $H'$ as obtained from $H$ by slowly expanding it
vertically. There is some first moment in time when $H'$ intersects $G$ or when $H'$ becomes non-embedded.

Let us suppose that $H'$ intersects $G-H$ first.  Emanating from this point of intersection, there is a geodesic segment which runs vertically in $H'$ to the boundary of the cusp $H$.  Pick a lift $\alpha$ of this geodesic segment in $\HH^2$ with one endpoint in $\partial H_\infty$. Let $x$ be the other endpoint. Note by choice of $H'$, the geodesic segment $\alpha$ has length at most $d$.  The starting point $x$ of $\alpha$ lies on a lift $\widehat g$ of a component of $G$.  This bi-infinite geodesic $\widehat g$ has ends in horoballs $H_1$ and $H_2$ that are not $H_\infty$.  The length of $\widehat g$ between $H_1$ and $H_2$ is at most $2d$, and so one of the components $\beta$ of $\widehat g - x$ has length at most $d$ outside of $H_1$, say.  Extend $\alpha$ and $\beta$ to infinite geodesic rays $\overline{\alpha}$ and $\overline{\beta}$ running into $H_\infty$ and $H_1$, respectively, and meeting at $x$. Note that  $\overline{\beta}$ misses $H_\infty$, because otherwise, by Claim 2, the entire vertical projection of $\widehat g$ onto $\partial H_\infty$ has image in $I$, whereas, by assumption, $\alpha \cap \partial H_\infty$ does not have image in $I$.
See Figure \ref{fig:expandhoroball2}. Straighten $\overline{\alpha}\cup \overline{\beta}$ to a geodesic $\gamma$ that hits $\partial H_\infty$ and $\partial H_1$ orthogonally. Clearly, the image of $\gamma$ in $F$ has length at most $2d$ with respect to $H$. Also, the image of $\gamma$ is embedded in $F$, because it is a geodesic homotopic to $\overline{\alpha} \cup \overline{\beta}$, which has image embedded in $F$. So the image of $\gamma$ lies in $G$, by the maximality of $G$.  Consider the triangle formed by $\gamma$, $\overline{\alpha}$ and $\overline{\beta}$. This has area less than $\pi$.  Since $\alpha$ avoids $I(\gamma)$, and $\overline{\beta}$ misses $H_\infty$, we deduce that the region lying between $\overline{\alpha}$ and $\gamma$ in $H_\infty$ has area at least $\pi$, which is a contradiction, proving the claim in this case.

\begin{figure}
  \includegraphics[width=2.5in]{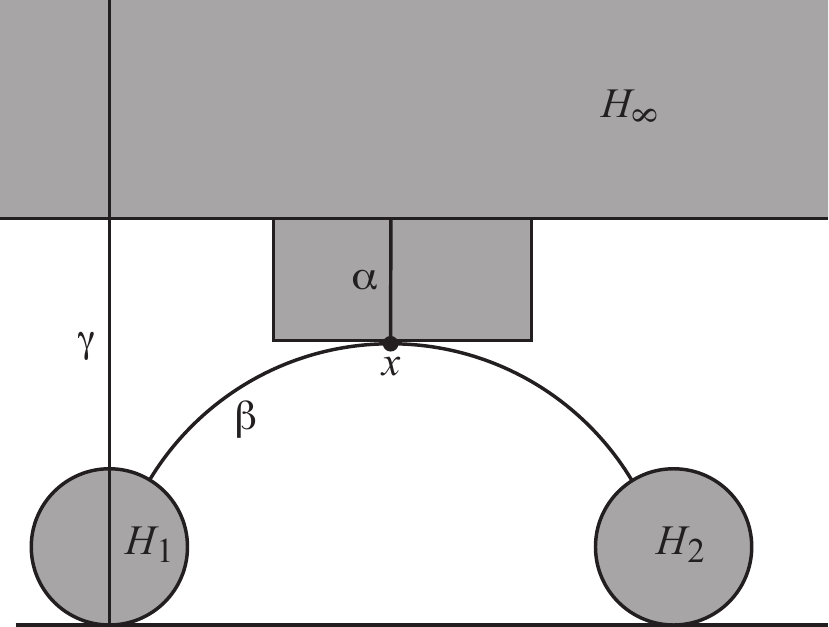}
  \caption{The case where $H'$ intersects $G-H$ first}
  \label{fig:expandhoroball2}
\end{figure}

Suppose now that when $H'$ is formed by expanding $H$, it becomes non-embedded first.  Then at this moment in time, the neighbourhood fails to be embedded at some point.  Emanating from this point, there are two geodesic segments which run to the boundary of the horoball neighborhood $H$, and which lift to geodesic segments $\alpha_1$ and $\alpha_2$ in $\HH^2$ running to the boundary of distinct horoballs in $\HH^2$.  We may take one of these to be $H_\infty$.  The geodesic segments $\alpha_1$ and $\alpha_2$ each have length at most $d$, by choice of $H'$.  So extending to infinite rays $\overline{\alpha}_1$ and $\overline{\alpha}_2$ and concatenating, we obtain a piecewise geodesic that has length at most $2d$ with respect to lifts of $H$.  Its image in $F$ is embedded (because this moment in time was the first time when $H'$ bumped into itself). It also misses $G$.  So, it is isotopic to an embedded geodesic $g$ that has length at most $2d$ with respect to the cusp.  Thus, $g$ lies in $G$ by the maximality of $G$.  So at the two points where $g$ intersects $\partial H$, there are the images of the two intervals $I(\widetilde g)$ and $I(\widetilde g')$.
Now, neither of these intervals is expanded when forming $H'$, and so at least one of $\overline{\alpha}_1$ and $\overline{\alpha}_2$ misses both intervals, say $\overline{\alpha}_1$ does. Consider the triangle formed by $g$, $\overline{\alpha}_1$ and $\overline{\alpha}_2$.  This has area less than $\pi$.  But the region lying between $\overline{\alpha}_1$ and $g$ inside $H$ has area at least $\pi$.  This contradiction proves the claim.

\medskip

The claim implies the theorem, because then
\begin{align*}
\area(H') & \geq 
\area(H) + (e^d - 1) \length(\partial H - I) \\
& \geq  \area(H)+(e^d - 1)\length(\partial H)  - (e^d - 1)|G|(2 \sinh (d) + 4 \pi)
\\
& =  e^d \area(H) - (e^d - 1) |G|(2 \sinh (d) + 4 \pi).
\end{align*}
Thus
\begin{align*}
\area(F) & \geq  \area(H') \\
& \geq  e^d \area(H) - (e^d - 1) |G| (2 \sinh (d) + 4 \pi)\\
& =  k e^d \area(F) - (e^d - 1) |G| (2 \sinh (d)  + 4 \pi).
\end{align*}
So,
\begin{align*}
|G| & \geq \frac{k e^d - 1}{(e^d - 1)} \frac{\area(F)}{(2 \sinh (d)  + 4 \pi)} \\
& = \frac{(k e^d - 1)\pi }{(e^d - 1)(\sinh (d) + 2 \pi)} |\chi(F)|.
\end{align*} \qedhere
\end{proof}


Recall from Definition \ref{def:length-arc}, the definition of the length of an arc in $S^3 \setminus K$, measured with respect to a maximal cusp.

\begin{lemma}\label{lemma:shortarcs-volume}
Suppose that a one--cusped hyperbolic 3--manifold $M$ contains at least $p$ homotopically distinct essential arcs, each with length at most $L$ measured with respect to the maximal cusp $H$ of $M$.  Then the cusp area $\area(\bdy H)$ is at least  $p\,\sqrt{3}\,e^{-2L}$.
\end{lemma}

\begin{proof}
Consider the universal cover $\HH^3$ of $M$, arranged so that one component of the lift of the maximal cusp $H$ is the horoball $H_\infty = \{z \geq1\}$.  Let $T$ be a fundamental domain for the cusp torus $\bdy H$ on $\bdy H_\infty$.

Let $G$ be our collection of essential arcs, with $|G| \geq p$.  Each arc $a$ in $G$ has geodesic representative in $M$ which lifts to give \emph{two} vertical geodesics meeting $T$: for oriented $a$, one such lift runs from infinity to a point on $\CC$ below $T$, and the other lift runs from a point on $\CC$ below $T$ to infinity.  One end of each arc is at infinity.  The other runs into a horoball centered at some other point on $\CC$.  Since the maximal cusp is embedded, each of these horoballs is disjoint in $\HH^3$.  Hence each gives a circular ``shadow'' on $\partial H$.  We estimate the size of the shadows and apply a circle packing argument to obtain our area estimate.

Because the (hyperbolic) length of each arc $a$ is at most $L$, the horoball at the end of each lift of $a$ in $T$ has (Euclidean) diameter at least $e^{-L}$.  Each such horoball projects onto $\bdy H_\infty$, giving a disk of diameter $e^{-L}$ on $\partial H$.  Note that while larger horoballs may partially cover smaller ones, if we shrink all horoballs to have diameter exactly $e^{-L}$, then the collection of horoballs will remain embedded in $\HH^3$, and so the projections of these smaller disks onto $\partial H$ are all disjoint.  Hence each arc $a$ in $G$ corresponds to two embedded disks on $\partial H$ of diameter $e^{-L}$.

It follows that the area of the cusp torus $\partial H$ is at least $2p \cdot \pi(e^{-L}/2)^2 = p\,\pi e^{-2L}/2$.  We can improve this estimate using a disk packing argument (see \cite[Theorem 1]{boroczky}): the area of must be at least $2\sqrt{3}/\pi$ times the sum of the areas of the disks, or
\[ \area(\bdy H) \geq \frac{2\sqrt{3}}{\pi}\cdot \frac{p\, \pi e^{-2L}}{2} = p\,\sqrt{3}\, e^{-2L}. \qedhere \]
\end{proof}

\begin{theorem}\label{thm:bounded-crossings}
Suppose $K$ is a hyperbolic alternating knot with a prime, twist reduced diagram, which we also call $K$, with no more than $N$ crossings in each twist region.  Let $H$ be a maximal cusp of the hyperbolic manifold $S^3 \setminus K$.  Then the cusp area satisfies:
\[ \frac{1.844 \times 10^{-4}}{(3\, N - 1)}\,(\tw(K)-2) \leq \area(\bdy H) <  10\sqrt{3}\, (\tw(K)-1). \]
\end{theorem}

\begin{proof}
The upper bound comes from the upper bound on the hyperbolic volume of an alternating knot complement, proved by Lackenby and improved by Agol and D.~Thurston in \cite{lackenby:volume-alt}.  From this result, we know
\[\vol(S^3\setminus K) \leq 10\, v_3 \,(\tw(K) -1),\]
where $v_3 \approx 1.01494$ is the volume of a regular hyperbolic ideal tetrahedron.
  
B\"or\"oczky~\cite{boroczky} proved that the cusp volume of a hyperbolic 3--manifold is at most $\sqrt{3}/(2\, v_3)$ times the volume of the manifold. Hence we have $\vol(H) < 5\sqrt{3}\, (\tw(K) -1).$ Finally, by basic hyperbolic geometry, we have $\area(\bdy H) = 2\vol(H)$.

For the lower bound, we first note that it trivially holds in the case of the figure-eight knot and the knot $5_2$, because these both have $\tw(K) = 2$. So, we can assume that $K$ is neither of these knots. Let $S$ denote the disjoint union of the two checkerboard surfaces.  Since $S$ is essential, it may be pleated in $S^3\setminus K$, and the pleating pulls back to give a hyperbolic metric on $S$.  Let $H_S$ be an embedded cusp for $S$ coming from Lemma~\ref{lemma:checkerboard-cusparea}, with $\area(H_S) \geq 2^{5/4}\,\Cr(K)$. By Lemma \ref{lemma:checkerboard-area}, $S$ has area $2\pi(\Cr(K) - 2))$. Set $k = \area(H_S)/\area(S)$. Then, $k > \root 4 \of 2 / \pi$. Set $d = \log(2/k)< \log(2^{3/4} \pi)$. Then by Theorem \ref{thm:arcs}, there are at least 
$$\frac{(k e^d - 1)\pi }{(e^d - 1)(\sinh (d) + 2 \pi)} |\chi(S)| > (0.083) \, |\chi(S)|$$
disjoint geodesic arcs in $S$ of length at most $2d$ with respect to $H_S$. Map these into $S^3\setminus K$.  Because the checkerboard surfaces are essential, each such arc remains essential under the mapping. Moreover, the length of each such arc may only decrease in $S^3\setminus K$, either by straightening, or by cutting off more length since the image of $H_S$ is contained in $H$.  Hence each has length at most $2d$ with respect to the maximal cusp of $S^3\setminus K$.

By Corollary \ref{corollary:numberhomotopicarcs}, at most $3N-1$ of these arcs may be homotopic to the same arc. Thus in $S^3\setminus K$, we have at least 
$(0.083) |\chi(S)|/(3N-1)$ homotopically distinct essential arcs of length at most $2d$.

Now Lemma \ref{lemma:shortarcs-volume} implies that the cusp area satisfies:
\begin{align*}
\area(\bdy H) &\geq \frac{(0.083) \, |\chi(S)|}{3N -1} \, \sqrt{3} \, e^{-4d}\\
& >  \frac{(0.083) \, \sqrt{3}  \, |\chi(S)|}{8\, \pi^4(3\, N - 1)}\\
& \geq \frac{1.844 \times 10^{-4}}{(3\, N - 1)}\,(\Cr(K)-2) \\
& \geq \frac{1.844 \times 10^{-4}}{(3\, N - 1)}\,(\tw(K)-2).\qedhere
\end{align*}
\end{proof}

\section{An essential twisted surface}\label{sec:twisted}
In the previous section, we showed that the complement of a hyperbolic alternating knot contains essential arcs of bounded length by considering the checkerboard surfaces of the knot.  We want the number of such arcs to be bounded below by a linear function of the twist number of the knot. When there is a bounded number of crossings per twist region, then the checkerboard surfaces give us this result.  However, in general, if there are a high number of crossings in a twist region, then the arcs with bounded length on the checkerboard surface may all be homotopic within this twist region.  We need to find arcs that we can guarantee are not all concentrated within a constant number of twist regions.

In order to do so, we use an essential immersed surface in the knot complement, called a \emph{twisted checkerboard surface}.  Two variations of this surface are defined in \cite{lp:twcheck}, and proved to be $\pi_1$-injective and boundary-$\pi_1$-injective.  
A map $f \colon S \rightarrow M$ of a surface $S$ into a compact 3-manifold $M$
with $f(\partial S) \subset \partial M$ is said to be \emph{boundary-$\pi_1$-injective} if, for any arc $\alpha \co I \to S$ with endpoints in $\partial S$, 
the existence of a homotopy (rel endpoints) of $f \circ \alpha$ into $\partial M$ implies the existence of a homotopy
(rel endpoints) of $\alpha$ into $\partial S$.

In this section, we review the definitions of these surfaces, and remind the reader of key results in \cite{lp:twcheck}.  Using these results, we give a proof of Theorem \ref{thm:main}.

\subsection{Twisted surfaces}
Recall that $K$ is a hyperbolic alternating knot with a prime, twist reduced diagram, which we also refer to by $K$.  For each twist region of $K$ with at least $N$ crossings ($N$ will be determined explicitly below), augment the diagram, in the sense of Adams \cite{adams:aug}. That is, add an unknotted circle to the diagram encircling the two strands of the twist region.  This circle is called a \emph{crossing circle}.  It bounds a disk, the \emph{crossing disk}, which meets the knot transversely in two points.

The result of adding crossing circles to all twist regions with at least $N$ crossings yields the diagram of a link.  We will denote this link and its diagram by $L$.

Recall that $S^3\setminus L$ is homeomorphic to $S^3\setminus L'$, where $L'$ is any link obtained from $L$ by removing any even number of crossings from each twist region encircled by a crossing circle in the diagram of $L$ (e.g.\ \cite{rolfsen}).  We will work with the diagram in which each twist region encircled by a crossing circle has $1$ or $2$ crossings remaining.  We denote this link (and its diagram) by $L_2$.  Finally, for $L_2$, remove the crossing circles from the diagram.  The result is a knot which we denote by $K_2$.  An example is shown in Figure \ref{fig:links}.

\begin{figure}
  \includegraphics{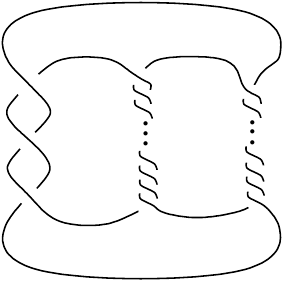}
  \hspace{.2in}
  \includegraphics{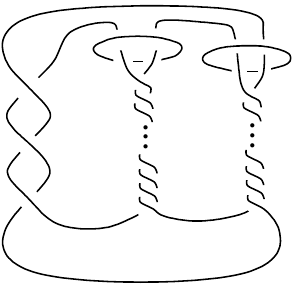}
  \hspace{.2in}
  \includegraphics{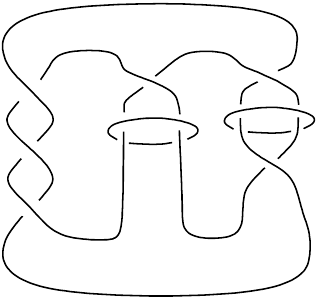}
  \hspace{.2in}
  \includegraphics{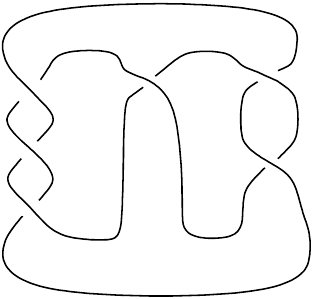}
  \caption{Examples of diagrams for $K$, $L$, $L_2$, and $K_2$. Figure from \cite{lp:twcheck}.}
  \label{fig:links}
\end{figure}

To form the twisted checkerboard surface, start with the checkerboard surfaces of $K_2$.  These are embedded surfaces in $S^3\setminus \nu(K_2)$, where $\nu(K_2)$ denotes a small regular neighborhood of $K_2$.  Color these surfaces \emph{red} and \emph{blue}.

A small regular neighborhood of each crossing circle in $L_2$ will intersect either the red or the blue checkerboard surface in two meridian disks.  Thus when we remove these disks from the surfaces, we obtain two surfaces embedded in $S^3\setminus \nu(L_2)$, which we also color red and blue, and denote by $R_2$ and $B_2$.

Because $S^3\setminus \nu(L_2)$ is homeomorphic to $S^3\setminus\nu(L)$, the surfaces $R_2$ and $B_2$ are also embedded in $S^3\setminus\nu(L)$.  Notice that under the homeomorphism, a meridian curve on the boundary of a regular neighbourhood of the $j$-th crossing circle will be sent to a curve with slope $\pm 1/n_j$, where $2n_j$ crossings were removed from the corresponding twist region to go from the diagram of $L$ to that of $L_2$.

Finally, to obtain $S^3\setminus\nu(K)$ from $S^3\setminus\nu(L)$, we perform meridian Dehn filling on each crossing circle of $L$.  That is, we re-insert a regular neighbourhood of the crossing circle so that the meridian of the crossing circle bounds a disk in the attached solid torus.

To obtain the twisted checkerboard surfaces, when we re-insert the regular neighbourhood of the crossing circles, we also attach annuli or M\"obius bands to the surfaces $R_2$ and $B_2$, as follows.  Let $C_j$ denote the $j$-th crossing circle in $L$. Recall that one of $R_2$ or $B_2$ meets $\partial \nu (C_j)$ in two curves, corresponding to two meridians of a crossing circle of $L_2$. Because the boundary slopes are each $\pm 1/n_j$ on $\bdy\nu(C_j)$, a meridian disk of $\nu(C_j)$ meets the surface $2n_j$ times.  For each meridian disk, attach an interval $I$ to opposite points of intersection on the disk, ensuring that in $S^3\setminus\nu(K)$, the interval runs through the core of the solid torus $\nu(C_j)$.  When we take such an interval for each disk $D^2\times\{\theta\}$ in $D^2\times S^1$, the result is that we have attached an $I$--bundle to the surface $R_2$ or $B_2$. When $n_j$ is odd, so that we are attaching opposite boundary components when we attach the interval, the $I$--bundle is an annulus, connecting the two boundary components of $R_2$ (or $B_2$) on $\bdy\nu(C_j)$.  When $n_j$ is even, the intervals connect one boundary component on $\bdy\nu(C_j)$ to itself, hence we are attaching two M\"obius bands to $R_2$ (or $B_2$).  This is illustrated in Figure~\ref{fig:extend-in}.

\begin{figure}
  \includegraphics{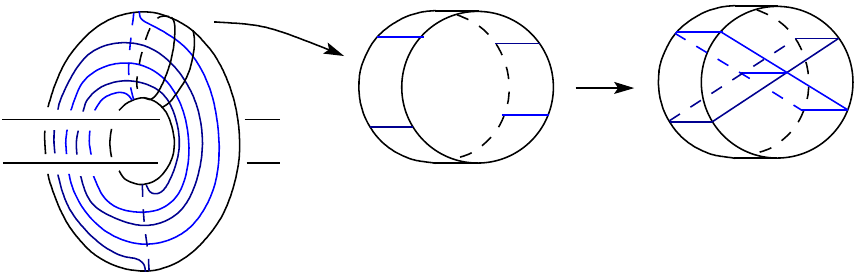}
  \caption{A cross section of the solid torus added to $S^3\setminus\nu(L)$, and how the surface extends into it.  Figure from \cite{lp:twcheck}.}
  \label{fig:extend-in}
\end{figure}

When we do this process for each crossing circle, we obtain two surfaces, denoted $S_{R,2}$ and $S_{B,2}$, and immersions $\phi_{R,2}\co S_{R,2} \to S^3\setminus K$ and $\phi_{B,2}\co S_{B,2}\to S^3\setminus K$.  Again we color the surfaces red and blue.

\begin{lemma}\label{lemma:euler-char}
Let $\tw_N(K)$ denote the number of twist regions of $K$ with at least $N$ crossings, and let $\Cr(\cdot)$ denote the number of crossings of a diagram.  Then the Euler characteristic of the red and blue twisted checkerboard surfaces satisfy
$$|\chi(S_{R,2}) + \chi(S_{B,2})| =  \Cr(K_2) + 2\,\tw_N(K) -2.$$
\end{lemma}

\begin{proof}
By Lemma \ref{lemma:checkerboard-area}, the absolute value of the Euler characteristic of the disjoint union of the checkerboard surfaces in $K_2$ is $\Cr(K_2)-2$.  To obtain surfaces $R_2$ and $B_2$, remove two disks for each crossing circle, so that $|\chi(R_2)+ \chi(B_2)| = \Cr(K_2) + 2\,\tw_N(K) - 2$.  Finally, connect pairs of boundary components by annuli or attach pairs of M\"obius bands for each crossing circle.  This has no effect on Euler characteristic.
\end{proof}

The next theorem is the main result of \cite{lp:twcheck}. 

\begin{theorem}\label{thm:tw-essential}
Let $S_{R,2}$ and $S_{B,2}$ be the surfaces constructed as above, by twisting checkerboard surfaces around crossing circles encircling each twist region with at least $N$ crossings.  Let $\phi_{R,2}\co S_{R,2} \to S^3\setminus K$ be the immersion of $S_{R,2}$ into $S^3\setminus K$, and similarly for $S_{B,2}$.  Then $\phi_{R,2}$ and $\phi_{B,2}$ are $\pi_1$-injective and boundary-$\pi_1$-injective, provided $N\geq 91$.
\end{theorem}

We also showed the following result, Theorem~7.1 in \cite{lp:twcheck}, analogous to Lemma \ref{lemma:checkerboard-htpcarcs}. As in Lemma \ref{lemma:checkerboard-htpcarcs}, this also refers to subsurfaces associated to twist regions, but we need to clarify what this means in the case of the twisted checkerboard surfaces. Each arises from a twist region of the knot $K_2$.
We choose the associated subsurfaces of the checkerboard surface for $K_2$ as follows. If the twist region is encircled by a crossing circle of $L_2$, we choose the subsurface in $R_2$ or $B_2$ so that it is punctured twice by this crossing circle. However, the subsurfaces are disjoint from all other crossing circles. Hence, 
they form subsurfaces of $R_2$ and $B_2$. Since there are inclusions $R_2 \subset S_{R,2}$ and $B_2 \subset S_{B,2}$, we obtain subsurfaces of $S_{R,2}$ and $S_{B,2}$ which are the \emph{subsurfaces associated with a twist region of $K_2$}.

\begin{theorem}\label{thm:htpcarcs}
Suppose $N\geq 121$ and suppose two distinct essential arcs in the surface $S_{B,2}$ have homotopic images in $S^3\setminus{K}$, but are not homotopic in $S_{B,2}$. Then the two arcs are homotopic in $S_{B,2}$ into the same subsurface associated with some twist region of $K_2$.
\end{theorem}

Just as in Corollary \ref{corollary:numberhomotopicarcs}, we obtain the following result.

\begin{corollary}\label{cor:numhtpcarcs}
Suppose $N\geq 121$. Let $\mathcal{C}$ be a collection of disjoint embedded essential non-parallel arcs in $S_{B,2} \sqcup S_{R,2}$ which are all homotopic in $S^3\setminus{K}$. Then the number of arcs in $\mathcal{C}$ is at most $2(3N-2)$.
\end{corollary}

\begin{proof}
By Theorem~\ref{thm:htpcarcs}, all the arcs of $\mathcal{C}$ that lie in $S_{B,2}$ can be homotoped in $S_{B,2}$ to lie in a suburface associated with a twist region of $K_2$. Note that after this homotopy, the arcs remain disjoint, embedded and non-parallel. One way of proving this is to use a hyperbolic structure on $S_{B,2}$, and to realise the arcs that separate the subsurface from the remainder of $S_{B,2}$ as geodesics, and also to realise the arcs in $\mathcal{C}$ as geodesics. Then the geodesics in $\mathcal{C}$ are disjoint, embedded and non-parallel, and lie in the subsurface associated with the twist region of $K_2$. The number of disjoint embedded essential non-parallel arcs in such a subsurface is at most $3N-2$. Note that when a twist region of $K_2$ is encircled by a crossing circle, it has at most 2 crossings, and so in this case, the maximal number of  disjoint embedded essential non-parallel arcs in a corresponding subsurface is at most $6$, which is less than $3N-2$. By switching roles of $S_{B,2}$ and $S_{R,2}$, we obtain the same result for all the arcs of $\mathcal{C}$ that lie in $S_{R,2}$. Hence there are at most $2(3N-2)$ arcs in total. 
\end{proof}

Because each twisted checkerboard surface is boundary-$\pi_1$-injective in $S^3\setminus K$, in the hyperbolic structure on $S^3\setminus K$ it can be given a pleating. (This is proved, for instance, in Lemma 2.2 of \cite{lackenby:word}.) The pleating pulls back to give a hyperbolic structure on the surface.  The following, analogous to Lemma~\ref{lemma:checkerboard-cusparea}, gives information on that hyperbolic structure.

\begin{prop}\label{prop:tw-cusp-area}
Let $S$ denote the disjoint union of $S_{B,2}$ and $S_{R,2}$, with immersion $f\co S\to S^3\setminus K$.  Consider the hyperbolic structure on $S$ given by pulling back a pleating of its image in the hyperbolic structure on $S^3\setminus K$.  Let $H$ be the maximal cusp for $S^3\setminus K$.  Then there is an embedded cusp for $S$ contained in the preimage $f^{-1}(H)$ with area at least $2\,\tw(K)$. Moreover, if $K$ is neither the figure-eight knot nor the knot $5_2$, then the area of this cusp is at least $2^{5/4} \, \tw(K)$.
\end{prop}

\begin{proof}
The boundary of the red surface on $\nu(K)$ is a single curve with integral slope, and similarly for the boundary of the blue surface.  The difference in slope $|\partial S_{B,2} - \partial S_{R,2}|$ is twice the crossing number of $K_2$, because the surfaces $R_2$ and $B_2$ intersect exactly in crossing arcs of $K_2$, with boundary slopes intersecting once at an overcrossing, and once at an undercrossing, just as in Figure~\ref{fig:crossingandsurface}.  Thus Lemma~\ref{lemma:union-cusparea} implies there is an embedded cusp for $S$ contained in the preimage $f^{-1}(H)$, with cusp area at least $2\,\Cr(K_2)$.  Moreover, if $K$ is neither the figure-eight knot nor the knot $5_2$, then the area of this cusp is at least $2^{5/4} \, \Cr(K_2)$.

Finally, because the diagram of $K_2$ contains at least one or two crossings per twist region of $K$, we obtain the required lower bound on cusp area.
\end{proof}

\subsection{Bounding cusp volume}

We are now ready to give the proof of the main theorem. 

\begin{named}{Theorem 1.1}
Let $D$ be a prime, twist reduced alternating diagram for some hyperbolic knot $K$, and let $\tw(D)$ be its twist number. 
Let $C$ denote the maximal cusp of the hyperbolic 3-manifold $S^3\setminus K$. Then 
$$A\,(\tw(D)-2) \leq \area(\bdy C) < 10\sqrt{3}\,(\tw(D)-1),$$
where $A$ is at least $2.278 \times 10^{-19}$.
\end{named}

\begin{proof}
The proof of the upper bound is identical to that of Theorem \ref{thm:bounded-crossings}, so we only need prove the lower bound. Again, this is trivial
in the case where $K$ is the figure-eight knot or the knot $5_2$, and so we will assume that it is neither of these knots.

Form the surfaces $S_{R,2}$ and $S_{B,2}$ as above, first by removing all but one or two crossings from twist regions of $K$ with at least $N=121$ crossings to yield the knot $K_2$, and then twisting the checkerboard surfaces of $K_2$ and attaching annuli and M\"obius bands.  Now let $S$ denote the disjoint union of $S_{R,2}$ and $S_{B,2}$, and let $f\co S \to S^3\setminus K$ denote the immersion. By Theorem \ref{thm:tw-essential}, $S$ is $\pi_1$-injective and boundary-$\pi_1$-injective, hence it can be pleated.
Give $S$ the hyperbolic metric induced by the pleating.

Let $H$ denote the maximal cusp of $S^3\setminus K$.  By Proposition \ref{prop:tw-cusp-area} there exists an embedded cusp $H_S$ for $S$ contained in the preimage $f^{-1}(H)$ such that $\area(H_S)\geq 2^{5/4}\,\tw(K)$. By Lemma \ref{lemma:euler-char},
$$\area(S) = 2 \pi (\Cr(K_2) + 2\,\tw_N(K) -2) < 240 \pi \, \tw(K).$$
Set $k = \area(H_S)/\area(S) \geq \root 4 \of 2 / (120 \pi)$, and set $d=\log(2/k)$.  By Theorem~\ref{thm:arcs}, there exist at least 
$$\frac{(k e^d - 1)\pi }{(e^d - 1)(\sinh (d) + 2 \pi)} |\chi(S)| > \frac{|\chi(S)|}{65143}$$
disjoint geodesic arcs in $S$ that have length at most $2d$ with respect to $H_S$. Map these arcs into $S^3\setminus K$.  Because $S$ is boundary-$\pi_1$-injective, each such arc remains essential under the mapping.  Moreover, its length may only decrease in $S^3\setminus K$, either by straightening the arc in $S^3\setminus K$, or by cutting off more length on one end, since $H_S$ may be a strict subset of $f^{-1}(H)$.  Hence each such arc has length at most $2d$ with respect to the maximal cusp of $S^3\setminus K$.

By Corollary~\ref{cor:numhtpcarcs}, for any subcollection $\mathcal{C}$ of these arcs, all of which are homotopic in $S^3 \setminus K$, the number of arcs in $\mathcal{C}$ is at most $2(3N-2) = 722$.
So, in $S^3\setminus K$, we will have at least $|\chi(S)|/(722 \times 65143) $ homotopically distinct embedded essential arcs of length at most $2d$.

Now, as in the proof of Theorem~\ref{thm:bounded-crossings}, Lemma~\ref{lemma:shortarcs-volume} implies that the cusp area satisfies:
\begin{align*}
\area(\bdy H) 
& \geq \frac{\sqrt{3}\,|\chi(S)|}{722 \times 65143} e^{-4d} \\
& \geq \frac{2\sqrt{3}}{(240 \,\pi)^4(722 \times 65143)}(\tw(K) -2) \\
& \geq 2.278 \times 10^{-19}(\tw(K) -2).\qedhere
\end{align*} 
\end{proof}

\bibliographystyle{amsplain}
\bibliography{references}

\end{document}